\documentclass[12pt]{article}
\usepackage{amsfonts,amsmath,amsthm, amssymb, amscd, mathrsfs}
\usepackage{latexsym, euscript, epic, eepic}
\usepackage{graphicx}
\usepackage{verbatim}
\newtheorem{lemma}{Lemma}[section]
\newtheorem{theorem}{Theorem}[section]
\newtheorem{definition}{Definition}[section]
\newtheorem{proposition}{Proposition}[section]
\newtheorem{corollary}{Corollary}

\newtheorem*{Claim}{Claim}
\newtheorem*{Lemma}{Lemma}
\newtheorem{remark}{Remark}

\newtheorem{theoremalph}{Theorem}

\begin{document}
\title{SRB measures for diffeomorphisms with continuous invariant splittings}
\author{Zeya Mi, \ Yongluo Cao and  Dawei Yang\footnote{D. Yang is the corresponding author. Y. Cao and D. Yang would like to thank the support of NSFC 11125103, NSFC 11271152 and A Project Funded by the Priority Academic Program Development of Jiangsu Higher Education Institutions(PAPD).}}
\maketitle
\begin{center}
\begin{minipage}{120mm}
{\small  {\bf Abstract}. We study the existence of SRB measures of $C^2$ diffeomorphisms for attractors whose bundles admit H\"older continuous
invariant (non-dominated) splittings.
We prove the existence when one sub-bundle has the \emph{non-uniform expanding} (the term was introduced in \cite{ABV00}) property
on a set with positive Lebesgue measure and
the other sub-bundle admits non-positive Lyapunov exponents on a total probability set.}
\end{minipage}
\end{center}

%

\section{Introduction}
Consider a smooth dynamical system $(M,f)$, where $M$ is a compact smooth manifold and $f$ is a $C^2$ diffeomorphism over $M$. Among all $f$-invariant Borel probabilities, we are interesting in finding measures that reflect the chaotic properties of $f$ from the viewpoints of entropies and Lyapunov exponents. In 1970s, Sinai, Ruelle and Bowen \cite{Bow75, BoR75, Rue76, Sin72} managed to get this kind of measures for hyperbolic systems.

Generally, for an invariant measure $\mu$ of $f$, if $(f,\mu)$ has a positive Lyapunov exponent and the conditional measures of $\mu$ along (Pesin) unstable manifolds of $\mu$ are absolutely continuous with respect to Lebesgue measures on these manifolds, then one says that $\mu$ is an \emph{SRB measure} (see for instance in \cite{y02} for this definition). Ledrappier-Young in \cite{ly} proved that this is equivalent to say that $h_\mu(f)$ is equal to the integral of the sum of its positive Lyapunov exponents, i.e., $(f,\mu)$ satisfies the Pesin Entropy formula.

One can ask the following question (by the philosophy of Palis \cite{pa}): what is the abundance of SRB measures for diffeomorphisms? In this paper, we will show the existence of SRB measures for systems with H\"older continuous invariant splittings and some weak hyperbolic properties.

Let $K\subset M$ be an attractor, i.e., $K$ is a compact invariant set and $K=\bigcap_{n\ge 0} f^n(U)$ for some open neighborhood $U$ of $K$ such that $\overline{f(U)}\subset U$. Assume that $T_{\overline{U}}M=E\oplus F$ is a H\"{o}lder continuous $Df$-invariant splitting. The simplest case is when
$U = M$. We say a Borel set has \emph{total probability} if it has measure one for every $f$-invariant probability measure.

\begin{theoremalph}\label{Theo-attractor}
Under the above setting, if we have
$$
{\rm Leb}\left(\left\{x\in U: \limsup_{n\rightarrow \infty}\frac{1}{n}\sum_{i=1}^{n}\log\|Df^{-1}/F(f^{i}(x))\|<0\right\}\right)>0,
$$
and there exists a subset $\Gamma\subset U$ with total probability such that for every point $x\in \Gamma$, it has
$$
\liminf_{n\rightarrow \infty}\frac{1}{n}\log\|Df^n/E(x)\| \le 0,
$$
then there is an SRB measure supported on $K$.

\end{theoremalph}
\begin{remark}
We only need the bundle $F$ to be \emph{H\"{o}lder} continuous in the proof.
\end{remark}
By adjusting the condition of $E$-direction, we have the next Corollary:

\begin{corollary}\label{corollary-attractor}

Under the assumption of theorem~\ref{Theo-attractor}, if we have
$$
\liminf_{n\rightarrow \infty}\frac{1}{n}\log\|Df^n/E(x)\|< 0
$$
on a set of total probability, then the SRB measure $\mu$ we get is \emph{physical}, in the following sense:

$${\rm Leb}\left(\left\{x:\frac{1}{n}\sum_{i=0}^{n-1}\delta_{f^i(x)}\xrightarrow{weak \ast}\mu\right\}\right)>0.$$

\end{corollary}

There are several previous related results. By the limit of our knowledge, we give a partial list below.

\begin{itemize}

\item Alves, Bonatti and Viana in \cite{ABV00} proved the existence of SRB (\emph{physical}) measures in ``mostly expanding" systems. Notice that all the splittings in \cite{ABV00} are \emph{dominated}. In contrast to \cite{ABV00}, the splitting in Theorem~\ref{Theo-attractor} is only H\"{o}lder continuous, which can be deduced by domination (one can see a proof in \cite[Theorem 3.7]{AP10}).

\item In  \cite{lep04} together with \cite{AL13}, the authors considered a system where the uniform hyperbolicity decreasing to vanish
when approaching to some invariant critical set. They assume there exists a subset $\Lambda$ of points that exhibit non
zero Lyapunov exponents (which ensures the orbits do not stay too long time in any fixed neighborhood of the critical set),
and have local stable/unstable manifolds with uniform size. These facts imply the existence of (countable) \emph{Markov Partition} over $\Lambda$.
By using the Markov Partition they proved that if there is an unstable manifold of some point in $\Lambda$ that intersects $\Lambda$ with positive Lebesgue measure,
then there exists some SRB measure.

\item Climenhaga, Dolgopyat and Pesin in~\cite{cdp13} considered a system with a measurable splitting and measurable invariant cone fields. They proved the system has some SRB measures if the system has some property called \emph{effective hyperbolicity} for measurable invariant cone fields. Unlike~\cite{cdp13}, systems in this paper do not have invariant cone fields.

\item One part of Theorem 1.2 of Liu and Lu in \cite{LiL15}  proved the existence of SRB measures for attractors with a continuous invariant splitting
with two bundles, one bundle is uniformly expanding and the other one has no positive Lyapunov exponents everywhere.
\end{itemize}

%
%
%

In this work, we need to deal with splittings which are not dominated. \cite{ABV00} studied the case of dominated splittings deeply.
 When we don't have the dominated property, we don't have the invariant cones generally, and we lose the estimations on the H\"older curvature
of sub-manifolds and the distortion bounds in contrast to \cite{ABV00}. However, the non-uniform expansion along $F$ and the non-expansion
along $E$ allow us to have some non-uniform domination on a set with positive Lebesgue measure. Then by focusing on
some special sets, we can recover the invariance of cones and the distortion bounds. By more accurate calculation,
we can estimate the H\"older curvature only for hyperbolic times on the special sets.

We also have a version for sub-manifold tangent to the $F$-bundle. Given sub-manifold $D$, denote by ${\rm Leb}_D$ the induced normalized Lebesgue measure restricted to $D$.
\begin{theoremalph}\label{Thm;submanifold}

Let $T_{\overline{U}}M=E\oplus F$ be a H\"{o}lder continuous $Df$-invariant splitting. If there is a $C^2$ local sub-manifold $D\subset U$, whose dimension is $\dim F$,
 such that
$$
{\rm Leb}_D\left(\left\{x\in D: ~T_x D=F(x), ~\limsup_{n\rightarrow \infty}\frac{1}{n}\sum_{i=1}^{n}\log\|Df^{-1}/F(f^{i}(x))\|<0\right\}\right)>0,
$$
and there exists a subset $\Gamma\subset U$ with total probability such that for every point $x\in \Gamma$, it has
$$
\liminf_{n\rightarrow \infty}\frac{1}{n}\log\|Df^n/E(x)\| \le 0,
$$
then there is an SRB measure supported on $K$.

\end{theoremalph}

Recall that R. Leplaideur \cite{lep04} considered some topologically hyperbolic diffeomorphisms.
 More precisely, \cite{lep04} discussed some open set $U$ containing a compact invariant set $\Omega$ with a H\"older
 continuous invariant splitting $E^{cs}\oplus E^{cu}$, together with continuous non-negative functions $k^s$ and $k^u$, such that
\begin{itemize}
\item $\|Df/E^{cs}(x)v\|\le {\rm e}^{-k^s(x)}\|v\|$, $\|Df/E^{cu}(x)v\|\ge {\rm e}^{k^u(x)}\|v\|$ for any $x\in U$ and
 any non-zero vector $v\in T_x M$ in the subspace, respectively;

\item $k^s(x)=0$ if and only if $k^u(x)=0$. Moreover, the set of all points with the above property is invariant.

\end{itemize}

\cite{lep04} proved that for some point with large unstable manifold and has good estimations, then $f$ admits a finite or $\sigma$-finite SRB measure. Especially, \cite{lep04} reduced the initial problem to \cite[Lemma 3.8]{lep04} which asserts that there is a unstable manifold $D$ such that
$$
{\rm Leb}_D\left(\left\{x\in D: ~\limsup_{n\rightarrow \infty}\frac{1}{n}\sum_{i=1}^{n}\log\|Df^{-1}/F(f^{i}(x))\|<0\right\}\right)>0.
$$

Our Theorem~\ref{Thm;submanifold} can apply to the main theorem of \cite{lep04} to show there really exists a finite SRB measure. Notice that this has already been obtained earlier by a recent paper of Alves-Leplaideur \cite{AL13}. However, the method here is different from \cite{AL13}: we don't need to estimate the unstable manifold in advance and we don't construct Markov partitions.

This paper is organized as follows. In Section $2$, we study the dynamics from continuous invariant splittings, mainly about some geometry properties for the iterated disks around some special points which will be denoted by $\Lambda_{\lambda_1,1}$, including the angles between these disks and $F$-bundle, the backward contracting property in hyperbolic times and also the bounded distortion property, and one should notice that this is the unique place using the \emph{H\"{o}lder} assumption of the $F$-bundle. Section $3$ is dedicated to prove Theorem \ref{Theo-attractor}, during which we will select some disk tangent to the $F$-direction cone field such that one can apply the properties obtained in Section $2$. Then we consider the Lebesgue measures of the disk under dynamics $f$ and we will find some ergodic measure of the accumulation of these measures satisfying Theorem \ref{Theo-attractor}. After that we will give the proof of Corollary \ref{corollary-attractor} as a simple application. Finally, a short proof of Theorem \ref{Thm;submanifold} will be presented by using the main approach we built in previous sections but with some modification.

\textbf{Acknowledgement} We would like to thank Prof. J. Xia for useful discussions and suggestions. Z. Mi would like to thank Northwestern University for their hospitality and the excellent research atmosphere, where this work was partially done. Z. Mi also would like to thank  China Scholarship Council (CSC) for their financial support.

\section{The dynamics from continuous invariant splittings}
Assume that $f$ is a $C^2$ diffeomorphism on a compact Riemmannian manifold  $M$, and $K$ is an attractor introduced in Section $1$. Let $T_{\overline{U}}M=E\oplus F$ be a $Df$-invariant continuous splitting throughout this section unless otherwise noted.

\subsection{Pliss Lemma and its applications}

The next classic Pliss lemma is very useful in getting hyperbolic times. The proof can be found for instance in \cite[Lemma 3.1]{ABV00}.

\begin{lemma}\label{Lem:numberPliss}(\textit{Pliss lemma})
For numbers $C_0\geq C_1>C_2\geq0,$ there is $\theta=\theta(C_0,C_1,C_2)>0$ such that for any integer $N$, and  numbers $b_1, b_2, \cdots ,b_N\in{\mathbb R}$, if
$$
\sum_{j=1}^N b_j\geq C_1N, ~~~b_j\leq C_0,~~~\forall~ 1\le j\le N.
$$
Then there is an integer $\ell>\theta N$ and a subsequence $1<n_1<\cdots <n_{\ell} \leq N$ such that
$$
\sum_{j=n+1}^{n_i} b_j\geq C_2(n_i-n) ~~~\text{for every} ~~0 \leq n<n_i ~~\text{and} ~~i=1,
\cdots ,\ell.
$$
\end{lemma}

We will use Lemma~\ref{Lem:numberPliss} to give some results for diffeomorphisms.

\begin{definition}(\textit{Hyperbolic time})
Given $\sigma <1$ and $x\in \overline{U} $, if
$$
\prod_{j=n-k+1}^n \|Df^{-1}/F(f^j (x))\| \leq \sigma^k,~~~\text{ for all}~~1 \leq k \leq n,
$$
then we say n is a $\sigma$-hyperbolic time for $x$.

\end{definition}

\begin{lemma}\label{Lem:hyperbolitimefordiff}
Given $0<\sigma_1<\sigma_2<1$, there is $\theta=\theta(\sigma_1,\sigma_2,f)>0$ such that for any $x$ and any $N\in{\mathbb N}$, if
$$\prod_{j=1}^N\|Df^{-1}/F(f^j(x))\|\le \sigma_1^N,$$
then, there are $1\le n_1<n_2<\cdots<n_\ell\le N$, where $\ell>\theta N$ such that $n_i$ is a $\sigma_2$-hyperbolic time for $x$, $1\le i\le \ell$.

\end{lemma}

\begin{proof}
This is an application of the Pliss Lemma (Lemma~\ref{Lem:numberPliss}) by taking $$b_j=-\log\|Df^{-1}/F(f^j(x))\|.$$

More precisely, by assumption,
we have
$$
\sum_{j=1}^N \log \|Df^{-1}/F(f^j(x))\|\le N\log \sigma_1.
$$
So
$$
\sum_{j=1}^N  b_j\geq(-\log \sigma_1)N.
$$
Now, we take $C_1=-\log \sigma_1,$ $C_2=-\log \sigma_2$ and $C_0=\sup \big|\log \|Df^{-1}/F\|\big|$. By the assumption, we have $C_0\geq C_1>C_2\geq0$ . Thus, Lemma~\ref{Lem:numberPliss} implies that there are $1\le n_1<n_2<\cdots<n_\ell\le N$ with $\ell>\theta N$ such that for every $n_i$, we have
$$
\sum_{j=n+1}^{n_i}  b_j\geq(-\log \sigma_2)(n_i-n) ~~~\text{for every}~~ 0\le n < n_i,
$$
which in other words,
$$
\prod_{j=n_i-k+1}^{n_i} \|Df^{-1}/F(f^j (x))\| \leq \sigma_2^k ~~~\text{for every}~~1 \leq k \leq n_i,
$$
that is to say $n_i$ is a $\sigma_2$-hyperbolic time for $x$, which completes the proof.

\end{proof}

We also need the following lemma of Pliss type which considers infinitely many times.

\begin{lemma}\label{Lem:infinitePliss}
For a sequence of real numbers $a_1, a_2, \cdots $, for $N\in \mathbb{N}$, if
$$
\sum_{i=1}^{n}a_i \geq 0,~~~ \text{for every }~n\ge N.
$$
Then there exists $1\leq k \le N$ such that
$$
\sum_{i=k}^n a_i \ge 0,~~~\text{for every }~n\ge k.
$$
\end{lemma}
\begin{proof}
Denote by $S(n)=\sum_{i=1}^n a_i$, for every $n\geq 1$. By convention, one can define $S(0)=0$. By the hypothesis, $S(n)\geq 0$ for every $n\ge N $. We choose some $0 \le \ell \leq N$ such that $S(\ell)$ be the smallest number among the sequence of numbers $S(n)$, where $n$ takes from $0$ to $N$, that is
$$
S(\ell)=\min\{S(n): 0 \le n \le N\}.
$$
We can also restrict $0\le \ell \le N-1$ as $S(0)=0 \le S(N)$.
So $S(\ell) \le S(n)$ for every  $\ell < n \le N$, and also $S(\ell) \leq 0$
which implies that $S(n) \ge S(\ell)$ for every $n > N$.
Together, we obtain that $S(n) \ge S(\ell)$ for all $n > \ell$.
Now take $k=\ell +1$, we have $S(n) \ge S(k-1)$ for every $n\geq k$,
then
$$
\sum_{i=k}^n a_i =S(n)-S(k-1)\geq 0 ~~\text{for every}~ n \ge k.
$$
\end{proof}

Given $\lambda\in(0,1)$ and $N\in \mathbb{N}$, define
$$
\Lambda_\lambda=\left\{x\in U:~\limsup_{n\to\infty}\frac{1}{n}\sum_{i=1}^n\log\|Df^{-1}/F(f^i(x))\|\leq 2\log\lambda\right\};
$$

$$
\Lambda_{\lambda, N }=\left\{x\in \Lambda_\lambda: \frac{1}{n}\sum_{i=1}^{n}\log\|Df^{-1}/F(f^{i}(x))\|\leq\log \lambda<0,~~~ \forall n\geq N \right\}.
$$


As an application of Lemma \ref{Lem:infinitePliss}, we have the next Proposition which asserts that one can reduce the positive volume set in assumption of Theorem \ref{Theo-attractor} to some set $\Lambda_{\lambda,1}$ (non-invariant) also with positive volume. $\Lambda_{\lambda,1}$ can be manipulated easier.

\begin{proposition}\label{Pro:measureLebesgue}
Let
$$
\Lambda=\left\{x\in U: \limsup_{n\rightarrow \infty}\frac{1}{n}\sum_{i=1}^{n}\log\|Df^{-1}/F(f^{i}(x))\|<0\right\}.
$$
If ${\rm Leb}(\Lambda)>0$, then there exists a constant $\lambda\in(0,1)$, such that ${\rm Leb}(\Lambda_{\lambda, 1})>0$.

\end{proposition}

\begin{proof}
Let $$\Lambda(k)=\left\{x\in U:~\limsup_{n\to\infty}\frac{1}{n}\sum_{i=1}^n\log\|Df^{-1}/F(f^i(x))\|\leq-\frac{1}{k}\right\},$$
then, we have  $\Lambda=\bigcup_{k=1}^{\infty}\Lambda(k)$  by the definition of $\Lambda$. This together with ${\rm Leb}(\Lambda)>0 $ implies that ${\rm Leb} (\Lambda_{\lambda})>0 $ for some $\lambda\in (0,1)$. Notice that $\Lambda_{\lambda} =\bigcup_{N=1}^{\infty}\Lambda_{\lambda, N }$ and $\Lambda_{\lambda, N }\subset \Lambda_{\lambda, N+1 } $ for all $N \in \mathbb{N}$, so there exists an $N\in\mathbb{N}$ such that ${\rm Leb}(\Lambda_{\lambda, N })>0.$

For any $x \in \Lambda_{\lambda, N}$,
$$
\sum_{i=1}^n\left(-\log \|Df^{-1}/F(f^{i}(x))\|-\log \lambda^{-1}\right)\ge 0,~~~\forall n\ge N.
$$
Set $a_i=-\log \|Df^{-1}/F(f^{i}(x))\|-\log \lambda^{-1}$, then
$$
\sum_{i=1}^n a_i \ge 0, ~~~\forall n\ge N.
$$
By applying Lemma \ref{Lem:numberPliss}, there exists some $1 \le k \le N$ such that
$$
\sum_{i=k}^n a_i \ge 0, ~~~\forall n\ge k.
$$
Thus,
$$
\sum_{i=k}^n\left(-\log \|Df^{-1}/F(f^{i}(x))\|-\log \lambda^{-1}\right)\ge 0,~~~\forall n\ge k.
$$
So, by the definition of $\Lambda_{\lambda,1}$, we have $f^{k-1}(x)\in \Lambda_{\lambda,1} $.

Now we make a partition of $\Lambda_{\lambda,N}$, let
$$\Lambda_{\lambda,N,j}=\left\{x\in \Lambda_{\lambda,N}:\;f^j(x)\in\Lambda_{\lambda,1}\right\},$$
then $ \Lambda_{\lambda,N} = \bigcup_{j=0}^{N-1} \Lambda_{\lambda,N,j} $ . Since ${\rm Leb }(\Lambda_{\lambda,N})>0$, we have ${\rm Leb}(\Lambda_{\lambda, N,j})>0$ for some $0 \leq j\leq N-1$. The fact that $f^j(\Lambda_{\lambda,N,j})\subset \Lambda_{\lambda,1} $ implies ${\rm Leb}(\Lambda_{\lambda,1})>0 $.

\end{proof}

\begin{proposition}\label{pro:hyperbolictimeforgood}
Given $0<\sigma_1<\sigma_2<1$, there is $\theta=\theta(\sigma_1,\sigma_2,f)\in(0,1)$ such that
for every $x\in\Lambda_{\sigma_1,1}$ and any $N\in \mathbb N $, there are $1\le n_1<n_2<\cdots<n_\ell\le N$, where $\ell>\theta N$ such that $n_i$ is a $\sigma_2$-hyperbolic time for $x$, $i=1,\cdots ,\ell$.

\end{proposition}
\begin{proof}
For every $x\in \Lambda_{\sigma_1,1}$, by the definition we have that
$$
\frac{1}{n}\sum_{i=1}^{n}\log\|Df^{-1}/F(f^{i}(x))\|\leq\log \sigma_1<0,~~~ \forall n\geq 1,
$$
it is equivalent to

$$
\prod_{i=1}^N\|Df^{-1}/F(f^i(x))\|\le \sigma_1^N,~~~\text{for any }~~N\in{\mathbb N}.
$$
\\Now we can apply the Lemma~\ref{Lem:hyperbolitimefordiff} to end the proof.

\end{proof}

The above Proposition~\ref{Pro:measureLebesgue} and Proposition~\ref{pro:hyperbolictimeforgood} tell us that under the setting of Theorem~\ref{Theo-attractor} there exists a set with positive Lebesgue measure such that all the points there have infinitely many hyperbolic times, and these hyperbolic times have uniformly positive density.

\subsection{Adjusting constants}

We have the following theorem that asserts that if an iteration of a diffeomorphism $f$ has an SRB measure, then $f$ has an SRB measure itself. The proof is standard, hence omitted.

\begin{theorem}\label{thm:renormalized} For given $N\in \mathbb{N}$, if $\mu$ is an $SRB$ measure for $f^N$, then there exist some $SRB$ measure $\hat{\mu}$ for $f$. More precisely, we can take
$$\hat{\mu}=\frac{1}{N}\sum_{i=0}^{N-1}f_\ast^i\mu.$$
\end{theorem}

By Theorem~\ref{thm:renormalized}, for considering the existence of SRB measures, it suffices to consider $f^N$ for some integer $N$.

We need the following Proposition whose proof is similar to~\cite{cao03} and we omit it here also.

\begin{proposition}\label{weak contracting}
Let $\Lambda$ be a compact positively invariant set and $E\subset T_\Lambda M$ be a continuous $Df$-invariant bundle.
If there is a set $\Lambda'\subset \Lambda$ with total probability such that for any $x\in\Lambda'$, one has
$$
\liminf_{n\rightarrow \infty}\frac{1}{n}\log\|Df^n/E(x)\| \le 0
$$
then $\forall \varepsilon>0$, there exists  $N:=N(\varepsilon)\in \mathbb{N}$ such that
$$
\|Df^n/E(x)\| < {\rm e}^{n\varepsilon}
$$
for any $n \ge N$ and $x \in \Lambda$.
\end{proposition}


Thus, under the assumptions of the main theorems, by considering a large iteration of $f$, we can add some standing assumptions for $f$:

\begin{enumerate}

\item[H:]

there are constants $\varepsilon_0>0$, $\xi, \lambda_1, \lambda_2, \lambda_3 \in (0,1)$ such that
\begin{itemize}
  \item $\|Df/E(x)\|<{\rm e}^{{\varepsilon}_0}$ for every $x\in \overline{U}$;

  \item $0<\lambda_1<\lambda_1{\rm e}^{\varepsilon_0}<\lambda_2<\lambda_3={\lambda_2{\rm e}^{\varepsilon_0}}/{b^{\xi}}<1$, where\footnote{If $V_1,V_2$ are two $d$-dimensional linear space and $A:V_1\rightarrow V_2$ is a linear map, we define the \emph{mininorm}
$$
m(A)=\inf_{v\neq 0}\frac{\|Av\|}{\|v\|}.
$$
If the linear map $A$ is invertible, then one obtains $m(A)=\|A^{-1}\|^{-1}$.} $b=\inf_{x\in \overline{U}} m(Df/F(x)) >0$;

  \item ${\rm Leb}(\Lambda_{\lambda_1,1})>0$.
\end{itemize}


\end{enumerate}

For every $x\in \Lambda_{\lambda_1,1}$, by Proposition~\ref{pro:hyperbolictimeforgood} we know that there are infinitely many $\lambda_2$-hyperbolic times for $x$. Let $n$ be a $\lambda_2$-hyperbolic time for $x$, then by the definition of hyperbolic time and the standing assumption, we have
$$
\prod_{j=n-k}^{n-1}\frac{\|Df/E(f^j(x))\|}{m(Df/F(f^j(x)))} \le  ({e^{\varepsilon_0}\lambda_2})^k, ~~\text{for every}~~ 1 \le k \le n,
$$
furthermore, we get
$$
\prod_{j=n-k}^{n-1}\frac{\|Df/E(f^j(x))\|}{m(Df/F(f^j(x)))^{1+\xi}} \le \left (\frac{e^{\varepsilon_0}\lambda_2}{b^{\xi}}\right)^k={\lambda_3}^k,~~\text{for every}~~ 1 \le k \le n.
$$
Which means that if $n$ is a $\lambda_2$-hyperbolic time for $x$, we have
$$
\prod_{j=n-k}^{n-1}\frac{\|Df/E(f^j(x))\|}{m\left(Df/F(f^j(x))\right)^{1+\xi}} \le {\lambda_3}^k,~~\text{for every}~~ 1 \le k \le n.
$$

\subsection{Sub-manifolds tangent to cone field and their iterations}




Denote by $B_r(x)=\{y\in M: d(x,y)\leq r\}$ the closed ball of radius $r$ around $x$.

We can assume $M$ is an embedded manifold in $\mathbb{R}^N$ for $N$ large enough by the Whitney Embedding theorem.
For a subspace $A\subset \mathbb{R}^N $ and a vector $v\in \mathbb{R}^N $, writing $dist(v,A)=\min_{w\in A}\|v-w\|$ as the length of the distance between $v$ and its orthogonal projection to $A$. If $A,B$ are any two subspaces of $\mathbb{R}^N $, define the distance between them (see \cite[Chapter 2.3]{bp02} and \cite{byu87}),
$$
dist(A,B)=\max\left\{\max_{u\in A,\|u\|=1}dist(u,B),\max_{v\in B,\|v\|=1}dist(v,A)\right\},
$$
in particular, if subspaces $A$ and $B$ have the same dimension, we have
$$
\max_{u\in A,\|u\|=1}dist(u,B)=\max_{v\in B,\|v\|=1}dist(v,A).
$$


\begin{definition}(\textit{Cone field})
Let $0<a<1$, define the $F$-direction cone field $\mathcal{C}_a^F=\left(\mathcal{C}_a^F(x)\right)_{x\in U} $ of width $a$ by
$$
\mathcal{C}_a^F(x)=\Big\{v=v_E+v_F\in E(x)\oplus F(x) \; such \; that\; \|v_E\|\leq a\|v_F\|\Big\}.
$$
One can define the $E$-direction cone field $\mathcal{C}_a^E$ of width $a$ in a similar way.

\end{definition}
For an embedded sub-manifold $D$, we say that it is \emph{tangent to ${\cal C}_a^F$} if $T_x D\subset {\cal C}_a^F(x)$ for any $x\in D$.

%
%

If the splitting is \emph{dominated} as in \cite{ABV00}, then the $F$-direction cone field is invariant by $Df$. Now the splitting here is only \emph{continuous}.

In~\cite{ABV00}, the authors assume all the systems there have the dominated splitting property such that one can use the invariance
 property of the cones to obtain several nice geometry properties of the iterations of some embedded sub-manifold tangent to $F$-direction
cone field with small fixed width. Indeed, the invariance property ensures that all the images of these kinds of sub-manifolds are also
tangent to the $F$-direction cone field of the same width as before. Moreover, the angles between bundle $F$ and tangent spaces of these
iterated sub-manifolds are decreasing as iterated times increasing.

In our setting, due to the lack of domination we have no invariance property of the cone fields. Consequently, one can not iterate
every sub-manifold tangent to $F$-direction cone field such that all its iterations have the nice geometry properties like dominated case. However, it is enough for
us to iterate sub-manifold around the neighborhood of some particular points (points of $\Lambda_{\lambda_1,1}$).
For this reason, we shall study systems with domination in local sense. More precisely, we consider average dominated orbit segment (Definition \ref{Def;average dominated} below), and built the invariance property of
cones in weak sense, which means that for any disk containing a starting point of some orbit segment with average dominated, if it is tangent to $F$-direction cone field, then so do
their iterations whatever they admit some uniform small radius around the average dominated orbit. In fact, analogous to dominated case, the angles between $F$-bundle and iterated disks are decreasing exponentially. Furthermore, with the help of hyperbolic time
we can show that the iterated disks are backward contracted in exponential rate on hyperbolic times.


\begin{definition}\label{Def;average dominated} An orbit segment $\left(x,f^{n}(x)\right)$ is called $\gamma$-average dominated if for any $1\le i\le n$, we have
$$\prod_{j=0}^{i-1}\frac{\|Df/E(f^{j}(x))\|}{m\left(Df/F(f^{j}(x))\right)}\le \gamma^{i}.$$

\end{definition}

By the standing assumption, we have

\begin{lemma}\label{fromweakcontract}
For any point $x\in \Lambda_{\lambda_1,1}$, the orbit segment $\left(x, f^{n}(x)\right)$ is $\lambda_1{\rm e}^{\varepsilon_0}$-average dominated for any $n\in \mathbb{N}$.

\end{lemma}

The proof of this lemma is to use the definition directly.

Similar to the case of dominated splittings, we have the following two lemmas for an average dominated orbit segment.

\begin{lemma}\label{Lem:bowen-ball-dominated}

For any $0<\gamma_1<\gamma_2<1$, there is $r=r(\gamma_1,\gamma_2)>0$ such that for any $x$, if $\left(x,f^{n}(x)\right)$ is $\gamma_1$-average dominated, then for any $y\in U$ satisfying $d\left(f^{j}(x),f^{j}(y)\right)\le r$ for any $0\le j\le n-1$, one has that $\left(y,f^{n}(y)\right)$ is $\gamma_2$-average dominated.

\end{lemma}

\begin{proof} For any constants $0<\gamma_1<\gamma_2<1$, by the uniform continuity of $Df$ and bundles, there exists $r=r(\gamma_1,\gamma_2)>0 $ such that
$$
\sqrt{\gamma_1/\gamma_2}\leq \frac{\|Df/E(x)\|}{\|Df/E(y)\|}\le \sqrt{\gamma_2/\gamma_1};
$$
and
$$
\sqrt{\gamma_1/\gamma_2}\leq \frac{m(Df/F(x))}{m(Df/F(y))}\le \sqrt{\gamma_2/\gamma_1},
$$
whenever $d(x,y)\le r$.

Then, by hypothesis, we obtain the following:
\begin{eqnarray*}
\prod_{j=0}^{i-1}\frac{\|Df/E(f^{j}(y))\|}{m(Df/F(f^{j}(y)))}
&\le& \prod_{j=0}^{i-1}\frac{\sqrt{\gamma_2/\gamma_1}\|Df/E(f^{j}(x))\|}{\sqrt{\gamma_1/\gamma_2}m(Df/F(f^{j}(x)))}\\
&\le& \gamma_2^{i}
\end{eqnarray*}
for any $1\le i \le n$. That is to say, $\left(y,f^{n}(y)\right)$ is $\gamma_2$-average dominated.
\end{proof}

\begin{lemma}\label{Lem:cone-inside}

For any $\lambda\in(0,1)$ and $a\in(0,1)$, if $\left(x,f^{n}(x)\right)$ is $\lambda$-average dominated, then $Df^{i}(x){\cal C}_a^F(x)\subset {\cal C}_{\lambda^{i} a}^F\left(f^{i}(x)\right),$ for every $ 1\le i \le n $.

\end{lemma}

\begin{proof}
Denote $v_0\in {\cal C}_a^F(x)$ by $v_0=v_E+v_F$, where $v_E\in E(x)$, $v_F\in F(x)$ and $\|v_E\|/\|v_F\|\le a$. Since the orbit segment $\left(x,f^{n}(x)\right)$ is $\lambda$-average dominated,

\begin{eqnarray*}
\frac{\|Df^{i}(x)(v_E)\|}{\|Df^{i}(x)(v_F)\|}
&\leq& \prod_{j=0}^{i-1}\frac{\|Df/E(f^{j}(x))\|\|v_E\|}{m(Df/F(f^{j}(x)))\|v_F\|}\\
&\le & \lambda^i\frac{\|v_E\|}{\|v_F\|}\\
&\le & \lambda^i a,
\end{eqnarray*}
for any $1\le i\le n$.
By invariance, $Df^{i}(x)(v_E)\in E(f^{i}(x))$ and $ Df^{j}(x)(v_F)\in F(f^{i}(x))$, the above inequality means that $Df^{i}(x){\cal C}_a^F(x)\subset {\cal C}_{\lambda^{i} a}^F\left(f^{i}(x)\right),$ for any  $1\le i \le n $.

\end{proof}

As a consequence of Lemma~\ref{Lem:bowen-ball-dominated} and Lemma~\ref{Lem:cone-inside}, for sub-manifold tangent to $F$-direction cone field we have the following fact:

\begin{lemma}\label{Lem:submanifold-cone}

Given $\lambda\in(0,1)$, there exists $r>0$ such that for any $a\in (0,1)$, if $\left(x,f^{n}(x)\right)$ is $\lambda$-average dominated, then for any sub-manifold $D\ni x$ tangent to ${\cal C}_a^F$ such that $d_{f^{i}D}\left(f^{i}(x),\partial (f^{i}(D))\right)\le r$ for any $0\le i\le n-1$, then
\begin{itemize}
  \item $f^{j}(D)$ is tangent to ${\mathcal C}_{\lambda^{j/2} a}^F $ for any $1\le j\le n$;

  \item $dist(F(f^{j}(y),T_{f^{j}(y)}f^{j}(D))\le \lambda^{j/2} a$, for every $y\in D$ and $1\le j\le n$.
\end{itemize}

\end{lemma}

\begin{proof}

We take $r=r(\lambda,\lambda^{1/2})$ as in Lemma~\ref{Lem:bowen-ball-dominated}. For the sub-manifold $D$ and $v_0\in T_y D$ for any $y\in D$. Denote by $v_0=v_E+v_F$, where $v_E\in E(y)$ and $v_F\in F(y)$ satisfying $\|v_E\|/\|v_F\|\le a$. Since the orbit segment $\left(y,f^{n}(y)\right)$ is $\lambda^{1/2}$-average dominated by Lemma~\ref{Lem:bowen-ball-dominated}, we get the
first statement of this Lemma by applying Lemma~\ref{Lem:cone-inside} directly.

Now we will prove the second statement. Since $v_j=Df^{j}(y)(v_0)=Df^{j}(y)(v_E)+Df^{j}(y)(v_F)$ for every $1\le j\le n$, by the definition of \emph{dist} in previous remarks, we have
\begin{eqnarray*}
dist\left(\frac{Df^{j}(y)(v_F)}{\|Df^{j}(y)(v_F)\|},T_{f^j(y)}f^{j}(D)\right)
&\leq & \left\|\frac{Df^{j}(y)(v_F)}{\|Df^{j}(y)(v_F)\|}-\frac{v_j}{\|Df^{j}(y)(v_F)\|}\right\|                                 \\
& = &\frac{\|Df^{j}(y)(v_F)-v_j\|}{\|Df^{j}(y)(v_F)\|}\\
& = &\frac{\|Df^{j}(y)(v_E)\|}{\|Df^{j}(y)(v_F)\|}\\
& \leq &\lambda^{j/2} \cdot a.
\end{eqnarray*}
By the arbitrariness of $v_0$, we know
$$
dist \left(F(f^{j}(y)),T_{f^{j}(y)}f^{j}(D) \right)=\max_{w\in F(f^{j}y), \|w\|=1}dist(w, T_{f^{j}(y)}f^{j}(D))\leq \lambda^{j/2} \cdot a.
$$
\end{proof}

\begin{lemma}\label{Lemma:hyperbolictangentcone}
Given $0 <\gamma_1<\gamma_2<1$, there are $r_0>0$ and $a_0>0$ such that for any $r\in(0,r_0]$ and $a\in (0,a_0]$, if $\left(x,f^{n}(x)\right)$ is $\gamma_1$-average dominated and $n$ is a $\gamma_2$-hyperbolic time for $x$, then for any embedded sub-manifold $\widetilde{D}$ containing $x$ with radius larger than $r$ around $x$, there is a simply connected sub-manifold $D \subset \widetilde{D}$ containing $x$ as its interior such that
\begin{itemize}
  \item $f^{k}(D)\subset B_r\left(f^{k}(x)\right)$ for every $0 \le k \le n$;
  \item $f^{n}(D)$ is a disk of radius $r$ around $f^{n}(x)$;
  \item $d_{f^{n-k}D}\left(f^{n-k}(x), f^{n-k}(y)\right)\le (\gamma_2)^{k/2}d_{f^{n}D}\left(f^{n}(x),f^{n}(y)\right)$ for any point $y \in D$.
\end{itemize}
\end{lemma}
\begin{proof}
By the uniform continuity of $Df$ and the bundles, there exist constants $r_0>0$ and $a_0>0$, such that
$$
\frac{m(Df/{\widetilde F}(y))}{m(Df/F(x))} \geq \sqrt{\gamma_2}, \quad \eqno(1)
$$
whenever $d(x,y)\leq r_0$ and $dist(\widetilde F(y),F(y))\le a_0$, and also for every $n$, the orbit segment $\left(y,f^n(y)\right)$ is $\gamma_2$-average dominated whenever $d(f^ix,f^iy)\le r_0$ for any $0\le i \le n-1$.

For any $r\in (0,r_0]$ and $a\in (0,a_0]$ fixed, let $\widetilde{D}$ be an embedded sub-manifold satisfying $d_{\widetilde{D}}(x,\partial\widetilde{D})>r$. Define $D_i$ as the connected component of $f(D_{i-1})\cap B_r\left(f^i(x)\right)$ containing $f^i(x)$ inductively for any $1\le i\le n$, where $D_0$ is the connected component of $\widetilde{D}\cap B_r(x)$ containing $x$, and by construction we know $d_{D_0}(x,\partial D_0)\ge r$.

Now we will firstly show that the the sub-manifold $D_n$ contains some disk of radius $r$.
By the construction, we have $D_i\subset B_r(f^i(x))$ for $0\le i\le n$, and $f^{-k}(D_n) \subset D_{n-k}$ for every $0 \le k \le n$. Then Lemma~\ref{Lem:submanifold-cone} implies that all the pre-images $\{f^{-k}(D_n)\}_{0<k\le n}$ are tangent to the cone field with width $a$, respectively.
We will argue by absurd: we assume by contradiction that $D_n$ has radius less than $r$, then there exists some point $y_n\in \partial(D_n)$ such that $d_{D_n}(f^n(x),y_n)< r$. Define $y_{n-k}=f^{-k}(y_n)$ for every $0\le k \le n$, then $y_i\in D_i\subset B_r(f^i(x))$, for every $0\le i \le n-1$. Thus, we can choose a sequence of points $z_k\in f^{-(k+1)}(D_n)$ and apply the inequality $(1)$ to get the following estimation
\begin{eqnarray*}
d_{f^{-k}D_n}\left(f^{n-k}(x),y_{n-k}\right) &\ge & m(Df/T_{z_k}(f^{-k-1}D_{n})d_{f^{-k-1}D_n}(f^{n-k-1}(x),y_{n-k-1})\\
& \geq &\sqrt{\gamma_2}m(Df/F(f^{n-k-1}(x)))d_{f^{-k-1}D_n}(f^{n-k-1}(x),y_{n-k-1}),
\end{eqnarray*}
for every $0\le k \le n-1$. Consequently,
$$
d_{D_n}(f^n(x),y_n)\ge (\sqrt{\gamma_2})^k \prod_{j=n-k}^{n-1} m(Df/F(f^j(x)))d_{f^{-k}(D_n)}\left(f^{n-k}(x),y_{n-k}\right),
$$
for every $1\le k \le n$. As $n$ is a $\gamma_2$-hyperbolic time for $x$, we know
$$
\prod_{j=n-k}^{n-1}m(Df/F(f^j(x)))\ge \gamma_2^{-k}.
$$
So
$$
d_{D_n}\left(f^n(x),y_n\right) \ge (\gamma_2^{-1/2})^kd_{f^{-k}D_n}\left(f^{n-k}(x),y_{n-k}\right).\quad \eqno(2)
$$
By the assumption $d_{D_n}(f^n(x),y_n)<r$, we have all the points $y_i$ are contained in the interior of $B_r(f^i(x))$, and $d_{D_i}(f^i(x),y_i)\le {\gamma_2}^k\cdot r<r$, $0\le i \le n-1$. So $y_0 \in \partial(D_0)$ and $d_{D_0}(x,y_0)\ge r$, a contradiction.
Therefore, the radius of $D_n$ is larger than $r$. Consequently, we can take a disk $\widetilde{D_n}$ contained in $D_n$ with radius $r$.

Let $D=f^{-n}(\widetilde{D_n})$, then $D$ satisfies the first two properties by our construction immediately. The last inequality about the backward contracting property can be deduced
similarly to the process of the proof of inequality $(2)$, which is a consequence of the assumption that $n$ is a hyperbolic time for $x$ and the fact that sub-manifold $f^i(D)$ is tangent to the cone field (with width smaller than $a$) around $f^i(x)$ with radius not bigger than $r$, for every $0\le i\le n$. So we can apply the estimation $(1)$ inductively.
\end{proof}

\subsection{Distortion bounds and H\"older curvature at hyperbolic times}

\begin{proposition}\label{diskproperty}
There exist constants $a>0,r>0$ such that if $x\in \Lambda_{\lambda_1,1}$ and $n$ is a $\lambda_2$-hyperbolic time for $x$, for any sub-manifold $D$ tangent to $\mathcal{C}_a^F$ with radius larger than r around $x$, we have
\begin{itemize}
  \item $d_{f^{n-k}(D)}\left(f^{n-k}(x),f^{n-k}(y)\right)\le {\lambda_2}^{k/2}d_{f^n(D)}\left(f^n(x),f^n(y)\right)$ ~~for any  $0\le k \le n$;
  \item $dist\left(T_{f^j(y)}f^j(D),F(f^j(y))\right)\le {\lambda_2}^j \cdot a$ ~~for every $0\le j \le n$,
  \end{itemize}
whenever $y\in D$ such that $d_{f^n(D)}\left(f^n(x),f^n(y)\right)\le r$.
\end{proposition}

\begin{proof} It can be deduced from Lemma \ref{fromweakcontract} and Lemma \ref{Lemma:hyperbolictangentcone}.
\end{proof}


We will discuss \emph{bounded distortion}, it plays a crucial role in the proof of the existence of SRB measures. Now we use the assumption: $F$ is H\"{o}lder continuous.

\begin{proposition}\label{pro:bounded distortion}
There exist $a>0$, $r>0$ and $\mathcal{K}>0$ such that for any $C^1$ sub-manifold $D$ tangent to $\mathcal{C}_a^F$ with radius larger than r around $x\in \Lambda_{\lambda_1,1}$, and $n\geq 1$ is a $\lambda_2$-hyperbolic time for $x$, then
$$
\frac{1}{\mathcal{K}}\leq\frac{|\det Df^n/T_yD|}{|\det Df^n/T_xD|}\leq \mathcal{K}
$$
for every $y\in D$ such that $d_{f^nD}(f^nx,f^ny)\leq r $.
\end{proposition}

\begin{proof}
Choose $a,r$ that satisfy the condition of Proposition \ref{diskproperty}, without loss of generality, we suppose $r<1$, then one obtains
\begin{eqnarray*}
\left|\log \frac{|\det {Df}^n/T_yD|}{|\det{Df}^n/F(y)|}\right|& \leq & \sum_{i=0}^{n-1}\left|\log|\det Df/T_{f^iy}f^iD|-\log|\det Df/F(f^i(y))|\right|\\
& \leq & \sum_{i=0}^{n-1}R_1 dist(T_{f^iy}f^i D, F(f^i(y)))\\
& \leq & \sum_{i=0}^{n-1}R_1 {\lambda_2}^i a\\
& \leq & R_1 \cdot \frac{a}{1- \lambda_2},
\end{eqnarray*}
where $R_1$ is a universal constant depending only on $f$. Especially, by taking $x=y$, we have
$$
\left|\log \frac{|\det{Df}^n/F(x)|}{|\det {Df}^n/T_xD|}\right| \leq R_1\cdot \frac{a}{1- \lambda_2}.
$$
Since bundle $F$ is H\"{o}lder by assumption, we may suppose $x\mapsto F(x)$ is $\beta$-H\"{o}lder continuous
for some $0<\beta\leq 1$. Therefore, we have the following estimation
\begin{eqnarray*}
\left|\log \frac{|\det {Df}^n/F(y)|}{|\det {Df}^n/F(x)|}\right|& \leq &\sum_{i=0}^{n-1}\left|\log|\det Df/F(f^i(x))|-\log|\det Df/F(f^i(y))|\right|\\
& \leq & \sum_{i=0}^{n-1}R_2d(f^i(x),f^i(y))^{\beta}\\
& \leq & \sum_{i=0}^{n-1}R_2d_{f^iD}(f^i(x), f^i(y))^{\beta},
\end{eqnarray*}
\\where $R_2$ is the H\"older constant for $\log|\det Df/F|$. By Proposition~\ref{diskproperty},  we obtain
\begin{eqnarray*}
\left|\log \frac{|\det Df^n/F(y)|}{|\det Df^n/F(x)|}\right|& \leq &\sum_{i=0}^{n-1}R_2d_{f^iD}(f^i(x),f^i(y))^{\beta}\\
& \leq & R_2 \sum_{i=0}^{n-1}\left[(\lambda_2)^{\frac{n-i}{2}}d_{f^nD}(f^n(x), f^n(y))\right]^{\beta}\\
& \leq & R_2 \sum_{i=0}^{n-1}({\lambda_2}^{\frac{\beta}{2}})^{n-i}r^{\beta}\\
& \leq & R_2 \cdot \frac{{\lambda_2}^{\frac{\beta}{2}}r^\beta}{1-{\lambda_2}^{\frac{\beta}{2}}}.
\end{eqnarray*}
\\With all the inequalities above, it follows that
\begin{eqnarray*}
\left|\log\frac{|\det {Df}^n/T_yD|}{|\det {Df}^n/T_xD|}\right|&\leq& \left|\log \frac{|\det {Df}^n/T_yD|}{|\det {Df}^n/F(y)|}\right|+\left|\log \frac{|\det {Df}^n/F(y)|}{|\det {Df}^n/F(x)|}\right|\\
&+& \left|\log \frac{|\det {Df}^n/F(x)|}{|\det {Df}^n/T_xD|}\right|\\
& \leq & 2 R_1 \cdot \frac{a}{1- \lambda_2}+ R_2 \cdot \frac{{\lambda_2}^{\frac{\beta}{2}}r^\beta}{1-{\lambda_2}^{\frac{\beta}{2}}}.
\end{eqnarray*}
\\Now it suffices to take
$$\mathcal{K}=\exp\left(2R_1\frac{a}{1- \lambda_2}+ R_2 \cdot \frac{{\lambda_2}^{\frac{\beta}{2}}}{1-{\lambda_2}^{\frac{\beta}{2}}}\right). $$
\end{proof}

For an embedded $C^1$ sub-manifold $D$, we say this sub-manifold is $C^{1+\xi}$ or the tangent bundle $TD$ is $\xi$-H\"older continuous if $x\mapsto T_x D$ defines a H\"older continuous subsection (with H\"older exponent $\xi $) from $D$ to the Grassmannian bundles over $D$. 
We will discuss in local coordinates. By the compactness of $M$ we can choose $\delta_0 >0$ small and fixed in advance, such that for any $x\in M$ the inverse of exponential map $\exp_x^{-1}$ is well defined on the $\delta_0$ neighborhood of $x$. Denote by $V_x$ the corresponding neighborhood of the origin of $T_x M$, then we identify these two neighborhoods.

For every $a>0$, up to shrinking $\delta_0$ such that for any $y \in D \cap V_x$, $T_y D$ is parallel to a unique graph of some linear map $L_x(y)$ from $T_x D$ to $E(x)$, whenever $D$ is tangent to cone field $\mathcal{C}_{a}^F$. Now we can describe the H\"older property of tangent bundle in local coordinate form.

\begin{definition} For constants $C>0$ and $\xi \in (0,1]$ fixed in standing assumption $({\rm H})$, if $D$ is tangent to the cone field $\mathcal{C}_{a}^F$, we say that the tangent bundle $TD$ is $(C,\xi)$-H\"older continuous if
$$
\|L_x(y)\| \le Cd_D(x,y)^{\xi} ~~~\text{for every}~~y\in D\cap V_x.
$$
\end{definition}



Then, for given $C^{1+\xi}$ sub-manifold $D$ tangent to the $F$-direction cone field, we define its \emph{H\"older curvature}
$$
\mathcal{H}_c(D)=\inf \bigg\{C>0: TD~\text{is} ~(C,\xi)\text{-H\"older continuous} \bigg\}.
$$

In next section, we will iterate $C^2$ disks tangent to the $F$-direction cone field and then consider the limit condition of the iterated disks. The next Proposition makes one can apply the Ascoli-Arzela theorem to get the accumulated disks of hyperbolic times which we will prove are actually the unstable disks.
\begin{proposition}\label{Holder curvature}
There exist constants $0<\lambda_4<1$, $\mathcal{L}>0 $, $a>0$ and $r>0$ such that for any given $C^{1+\xi}$ sub-manifold $\widetilde{D}$ tangent to $\mathcal{C}_a^F$ with radius larger than r around $x\in \Lambda_{\lambda_1,1}$, if $n$ is a $\lambda_2$-hyperbolic time for $x$, then there is a sub-manifold $D\subset \widetilde{D} $ containing $x$ such that $f^n(D)$ is contained in $B_r(f^n(x))$ and the H\"older curvature of $f^n(D)$ satisfies
$$
\mathcal{H}_c(f^n(D))\le \lambda_4^n \mathcal{H}_c(\widetilde{D})+ \frac{\mathcal{L}}{1-{\lambda}_4}.
$$
As a consequence, $\mathcal{H}_c(f^n(D))< 2\mathcal{L}/ (1-{\lambda}_4)$  when the $\lambda_2$-hyperbolic time $n$ large enough.

\end{proposition}
\begin{proof}
By applying the Proposition~\ref{diskproperty}, we can choose $a>0$ and $r>0$, such that there exists a sub-manifold $D\subset \widetilde{D}$ containing $x$ with the following  properties:
\begin{itemize}
  \item $f^i({D})$ is contained in the corresponding ball of radius $r$, for any $0\le i \le n$;
  \item $f^n({D})$ is a disk of radius $r$ with center $f^n(x)$.
\end{itemize}
Without loss of generality we assume $r \le \delta_0$. Given $y\in D$, in the neighborhood $V_y$ and $V_{f(y)}$ we can express $f$ in local coordinate from $T_yD\oplus E(y)$ to $T_{f(y)}D\oplus E(f(y))$ as $f(u,v)=(u_1(u,v),v_1(u,v))$, then $Df(u,v)$ can be expressed by the following matrix

$$
Df(u,v)=
\left(
  \begin{array}{cc}
    \partial_u u_1 & \partial_v u_1 \\
    \partial_u v_1 & \partial_v v_1 \\
  \end{array}
\right),$$
as $Df(E(y))=E(f(y))$ and $Df(T_yD)=T_{f(y)}f(D)$, we have
$$\partial_u u_1(0,0)=Df/T_y D,~~ \partial_v u_1(0,0) =0,~~ \partial_u v_1(0,0)=0,~~ \partial_v v_1(0,0)=Df/E(y). $$

We have the following choices of constants:
\begin{itemize}

\item there is $L>0$ such that for any disk $D$ centered at $y$ tangent to the cone field associated to $F$, then $\|L_y(z)\|\le L$ for any $z\in D$. Clearly, we can assume $L\ge 1$.

\item $Df$ is $(L_1,\xi)$-H\"older.

\end{itemize}

Notice that the constants do not depend on $y$.
\bigskip

For every $0<\alpha< b/4$, we can adjust $r$, $a$ such that
\begin{eqnarray*}
&~&m(\partial_u u_1(z)) \ge m(Df/F(x))-\alpha/L,~\|\partial_v u_1(z)\| \le \alpha/L,\\
&~&\|\partial_u v_1(z)\|\le\alpha/L,~\|\partial_v v_1(z)-Df/E(x)\|\le \alpha/L,
\end{eqnarray*}
for any $z\in D$.

\begin{Claim}
The H\"older curvature ${\cal H}_c(f(D))$ of $f(D)$ has the following estimation:
$$
{\cal H}_c(f(D))\le \frac{\|Df/E(x)\|+2\alpha}{(m(Df/F(x))-2\alpha)^{1+\xi}}\mathcal{H}_c(D)+ \frac{L_1}{(m(Df/F(x))-2\alpha)^{1+\xi}}.
$$

\end{Claim}

\begin{proof}[Proof of the Claim]

For the estimation of the H\"older curvature of $f(D)$, it suffices to know
$$\sup_{z_1\in f(D)}\frac{\|L_{f(y)}z_1\|}{d_{f(D)}(f(y),z_1)^{-\xi}}$$
since one can choose $y\in D$ arbitrarily.

Now for every $z_1\in f(D)$, according to the previous argument there exists a unique linear map $L_{f(y)}(z_1)$ parallel to the tangent space $ T_{z_1}f(D)$, for pre-image $z$ of $z_1$, there also exists a unique linear map $L_y(z)$ parallel to $T_zD$, then by the Mean Value theorem we have that there exists some point $w \in D$ such that
$$
d_{f(D)}(f(y),z_1) \ge m(Df/T_w D)d_{D}(y,z)\ge (m(Df/F(x))-\alpha/L) d_{D}(y,z).
$$
By the construction, $L_{f(y)}$ has the following expression:
$$
L_{f(y)}(z_1)=\left(\partial_u v_1(z)+ \partial_v v_1(z)L_y(z)\right)\left(\partial_u u_1(z)+ \partial_v u_1(z) L_y(z)\right)^{-1}.
$$
We have $\|\partial_v u_1(z) L_y(z)\|\le \alpha/L \|L_y(z)\|\le \alpha < m(\partial_u u_1(z))$, and furthermore,
\begin{eqnarray*}
\|(\partial_u u_1(z)+ \partial_v u(z) L_y(z))^{-1}\| & \le & \frac{1}{m(\partial_u u_1(z))-\|\partial_v u(z)\|L}\\
& \le & \frac{1}{m(Df/F(x))-\alpha/L -\alpha}\\
& \le & \frac{1}{m(Df/F(x))-2\alpha},
\end{eqnarray*}

$$
\|\partial_u v_1(z)+ \partial_v v_1(z)L_y(z)\| \le L_1 d_D(y,z)^\xi+ (\|Df/E(x)\|+\alpha/L)\|L_y(z)\|.
$$

Combing all these estimations and the fact $\|L_y(z)\|{d_D(y,z)}^{-\xi} \le \mathcal{H}_c(D) $  we get that
\begin{eqnarray*}
\frac{\|L_{f(y)}z_1\|}{d_{f(D)}(f(y),z_1)^\xi}&\le&\frac{\|L_{f(y)}z_1\|}{( m(Df/F(x)-2\alpha)^{\xi}d_D(y,z)^\xi}\\
&\le&\frac{\|\partial_u v_1(z)+\partial_v v_1(z)\|}{ (m(Df/F(x))-2\alpha)^{1+\xi}d_D(y,z)^\xi}\\
&\le&\frac{L_1 d_D(y,z)^\xi}{(m(Df/F(x))-2\alpha)^{1+\xi}d_D(y,z)^\xi}+\frac{(\|Df/E(x)\|+2\alpha)L_y(z)}{(m(Df/F(x))-2\alpha)^{1+\xi}d_D(y,z)^\xi}\\
&\le&\frac{\|Df/E(x)\|+2\alpha}{(m(Df/F(x))-2\alpha)^{1+\xi}}\mathcal{H}_c(D)+ \frac{L_1}{(m(Df/F(x))-2\alpha)^{1+\xi}}.
\end{eqnarray*}

\end{proof}
Recall $b=\inf_{x\in \overline{U}} m(Df/F(x))$ and $\alpha< b/4$, we have
$$
\frac{L_1}{(m(Df/F(y))-2\alpha)^{1+\xi}}\le \frac{L_1}{{(b/2)}^{1+\xi}}.
$$
Define
$$
\mathcal{L}=\frac{2^{1+\xi}L_1}{b^{1+\xi}};
$$
and
$$
c_j=\frac{\|Df/E(f^j(x))\|+2\alpha}{(m(Df/F(f^j(x)))-2\alpha)^{1+\xi}} ~~\text{for every}~~0 \le j \le n-1.
$$

%
By using the claim inductively, we have that
$$
\mathcal{H}_c(f^n(D)) \le c_0\cdots c_{n-1} \mathcal{H}_c(D)+ {\mathcal{L}}(1+ c_{n-1}+c_{n-1}c_{n-2}+ \cdots + c_{n-1}\cdots c_1).
$$
Recall the comments after standing assumption $({\rm H})$, for some $\lambda_4 \in (\lambda_3,1)$  fixed in advance, by choosing $\alpha$ sufficiently small by reducing $r$ and $a$, thus we have the estimations
$$
\prod_{j=n-k}^{n-1} c_j \le \lambda_4^k ~~\text{for every}~~ 1 \le k \le n,
$$
then, we obtain
$$
\mathcal{H}_c({f^n(D)}) \le \lambda_4^n\mathcal{H}_c(D)+ \frac{\mathcal{L}}{1-\lambda_4}.
$$
\end{proof}

\section{The iteration of Lebesgue measure}
The main aim of this section is to prove Theorem~\ref{Theo-attractor}, by standing assumption $({\rm H})$ we know ${\rm Leb}(\Lambda_{\lambda_1, 1})>0$. Now we fix $a$ and $r$ as in Proposition~\ref{diskproperty} and Proposition~\ref{Holder curvature}. Then, reducing to a small neighborhood of some Lebesgue density point, one can construct a smooth foliation with all the leaves are smooth (so $C^2$) and tangent to the given cone field ${\cal C}_a^F$ everywhere.
Then there exists at least one leaf $D$ of this foliation such that $D$ intersects $\Lambda_{\lambda_1, 1}$ in a set of positive  Lebesgue measure by using the Fubini's theorem.

Now we consider the sequence of averages of forward iterations of Lebesgue measure restricted to the disk $D$ above, that is
$$
\mu_n=\frac{1}{n}\sum_{i=0}^{n-1}f_{\ast}^i \rm Leb_D~.
$$

In this section we will prove that there exists some ergodic component of any limit measure of $\mu_n$, which is the SRB measure in the Theorem \ref{Theo-attractor} or the \emph{Physical} measure in Corollary \ref{corollary-attractor}. Our main ideas in this section come from \cite{ABV00}.

\subsection{Construct absolute continuous (non-invariant) part of the limit measures}

For any disk $D$ containing $x$, denote by $B_D(x,\delta)$ the ball of radius $\delta$ around $x$ in $D$.

\begin{proposition}\label{Pro:acpart-iterate}
There are $\eta>0$ and $0<r_1<r$ such that for each $n$, there are points $x_{n,1},\cdots,x_{n,k(n)}\in f^n(D)$ such that
\begin{itemize}

\item $f^{-n}(x_{n,j})\in \Lambda_{\lambda_1,1}$ and $n$ is $\lambda_2$-hyperbolic time for $f^{-n}(x_{n,j})$ for $1\le j\le k(n)$;

\item $B_{f^n(D)}(x_{n,j},r_1/4),~1\le j\le k(n)$ are pairwise disjoint;

\item there is $\widetilde{\varepsilon}_0>0$ such that for any $\varepsilon\in[0,\widetilde{\varepsilon}_0)$, we have
$$\mu_{n,ac,\varepsilon}(\bigcup_{0\le i\le n-1}K_{i,\varepsilon})\ge \eta,$$
where
$$K_{n,\varepsilon}=\bigcup_{1\le i\le k(n)}B_{f^n(D)}(x_{n,i},\frac{r_1}{4}-\varepsilon);$$
$$\mu_{n,ac,\varepsilon}=\frac{1}{n}\sum_{i=0}^{n-1}\sum_{j=1}^{k(i)}f^i_*{\rm Leb}_D|B_{f^i(D)}(x_{i,j},\frac{r_1}{4}-\varepsilon).$$
We denote ${\mu}_{n,ac}=\mu_{n,ac,0}$ and $K_n=K_{n,0}$.

\end{itemize}
\end{proposition}

\begin{proof}
Take $r_1\in (0,r)$ such that if we let $D_0$ be a sub-disk of $D$ by removing the $r_1/2$ neighborhood of the boundary, then ${\rm Leb}_D(\Lambda_{\lambda_1,1} \cap D_0)>0 $. Define
$$
S_n=\bigg\{x \in \Lambda_{\lambda_1,1} \cap D_0: n~ \text{is a}~ \lambda_2 \text{-hyperbolic time for}~ x\bigg\}.
$$

\paragraph{Step 1:} First we will show that there exists a constant $\tau >0$ such that there are balls $B_{f^n(D)}(x_{n,j},r_1/4)$ for each $n$, where $x_{n,j}\in f^n(D)$ and  $1 \le j \le k(n)$, having the following properties:
\begin{itemize}

\item $f^{-n}(x_{n,j})\in \Lambda_{\lambda_1,1}$ and $n$ is a $\lambda_2$-hyperbolic time of $f^{-n}(x_{n,j})$ for $1\le j\le k(n)$;

\item $B_{f^n(D)}(x_{n,j},r_1/4),~1\le j\le k(n)$ are pairwise disjoint;

\item we have the estimation:

$$
f^n_\ast {\rm Leb}_D\left(\cup_{j=1}^{k(n)}B_{f^n(D)}(x_{n,j},r_1/4)\right) \ge \tau{\rm Leb}_D(S_n).\quad \eqno(3)
$$

\end{itemize}

Recall the Besicovitch Covering lemma, see \cite[2.8.9-2.8.14]{geometric}.

\begin{Lemma}\label{lem:Besicovitch}(\textit{Besicovitch Covering lemma})
For $k\in \mathbb{N}$, there exists constant $p=p(k)\in \mathbb{N}$ such that for any $k$ dimensional compact $C^2$ Riemannian manifold $N$, any set $A\subset N$, and for any family $\mathcal{B}$ of balls such that any $x\in A$ is in the central of some ball in $\mathcal{B}$, there exists a sub-families $\mathcal{B}_1,\cdots, \mathcal{B}_p$ contained in $\mathcal{B}$ with the following properties:
\begin{itemize}
  \item A $\subset \bigcup_{i=1}^{p}\bigcup_{B\in \mathcal{B}_i}B$;
  \item either $B \cap B'=\emptyset$, or $B=B'$, for any $B, B'$ in $\mathcal{B}_i$ and $1\le i \le p$.
\end{itemize}
\end{Lemma}

Now we shall apply the Besicovitch Covering lemma. For every fixed $n$, put $N=f^n(D)$, $A=f^n(S_n)$. $N$ is a $C^2$ sub-manifold since $f$ and $D$ are $C^2$. Denote by $\mathcal{B}=\{B_{f^n(D)}(x,r_1/4)$, $x\in A \}$ as the family of balls. As a consequence of Besicovitch Covering lemma, we can choose a sequence of sub-families $\mathcal{B}_1,\cdots, \mathcal{B}_p$ of $\mathcal{B}$ such that
$f^n(S_n) \subset \bigcup_{i=1}^{p} \bigcup_{B\in \mathcal{B}_i}B$ and every $\mathcal{B}_i$ is formed by disjoint balls with fixed radius $r_1/4$, so
$$
f^n_\ast {\rm Leb}_D\left(f^n(S_n)\right)\le f^n_\ast {\rm Leb}_D\left(\bigcup_{i=1}^{p} \bigcup_{B\in \mathcal{B}_i}B\right).
$$
We choose some $1\le i \le p$ such that
$$
f^n_\ast {\rm Leb}_D\left(\bigcup_{B\in \mathcal{B}_i}B\right) \ge \frac{1}{p} f^n_\ast {\rm Leb}_D\left(f^n(S_n)\right)=\frac{1}{p}{\rm Leb}_D(S_n).
$$
Let $B_{f^n(D)}(x_{n,j},r_1/4)$, $1\le j \le k(n)$ be the disjoint balls of $\mathcal{B}_i$. Then by our construction $f^{-n}(x_{n,j})\in \Lambda_{\lambda_1,1}$ and $n$ is the $\lambda_2$-hyperbolic time for $f^{-n}(x_{n,j}), 1\le j \le k(n)$, the above estimation becomes
$$
f^n_\ast {\rm Leb}_D\left(\cup_{j=1}^{k(n)}B_{f^n(D)}(x_{n,j},r_1/4)\right) \ge \frac{1}{p}{\rm Leb}_D(S_n).
$$
It suffices to take $\tau=1/p$ to end this step.

\paragraph{Step 2:}Define
$$\mu_{n,ac}=\frac{1}{n}\sum_{i=0}^{n-1}\sum_{j=1}^{k(i)}f^i_*{\rm Leb}_D|B_{f^n(D)}(x_{i,j},r_1/4).$$
We consider the space $\{0,1,\cdots,n-1\}\times D$ with the product measure $\xi_n\times {\rm Leb}_D$, where $\xi_n$ is the uniform distribution measure on $\{0,1,\cdots,n-1\}$.  Define the indicator function
$$
\chi(x,i)=\left\{
\begin{array}{cr}
1 & ~~~\text{if}~x\in S_i~;\\
0 & ~~~\text{otherwise}~.
\end{array}\right.
$$

Then, by using  Fubini's theorem
\begin{eqnarray*}
\frac{1}{n}\sum_{i=0}^{n-1}{\rm Leb}_D(S_i)&=&\int\left(\int\chi(x,i)d{\rm Leb}_D(x)\right)d\xi_n(i)\\
&=& \int\left(\int\chi(x,i)d\xi_n(i)\right)d{\rm Leb}_D(x).
\end{eqnarray*}
Since ${\rm Leb}_D( \Lambda_{\lambda_1,1} \cap D_0)> 0$ and the density of $\lambda_2$-hyperbolic times for all the points in $\Lambda_{\lambda_1,1}$ are bounded from below by $\theta=\theta(\lambda_1,\lambda_2,f)>0$. In other words, $\int\chi(x,i)d\xi_n(i)\ge\theta$. Thus
$$\frac{1}{n}\sum_{i=0}^{n-1}{\rm Leb}_D(S_i)\geq \theta {\rm Leb}_D( \Lambda_{\lambda_1,1} \cap D_0),~~~\text{for every }~n\in \mathbb{N}. \quad \eqno(4)
$$
 By the definition of $\mu_{n,ac}$, and the estimations from $(3)$ and $(4)$, it follows that
\begin{eqnarray*}
\mu_{n,ac}\left(\bigcup_{0\le i\le n-1}\bigcup_{1\le j\le k(i)}B_{f^i(D)}(x_{i,j},r_1/4)\right)& \ge & \frac{1}{n}\sum_{i=0}^{n-1}\sum_{j=1}^{k(i)}(f_*^i {\rm Leb}_D)\left(B_{f^i(D)}(x_{i,j},r_1/4)\right)\\
&=&\frac{1}{n}\sum_{i=0}^{n-1}(f_*^i {\rm Leb}_D)(K_i)\\
&\geq & \frac{1}{n}\sum_{i=0}^{n-1}\tau {\rm Leb}_D(S_i)\\
&\geq & \tau \theta {\rm Leb}_D(\Lambda_{\lambda_1,1} \cap D_0),
\end{eqnarray*}
then, by definition, we have $$\mu_{n,ac}\left(\bigcup_{0\le i\le n-1}K_i\right)=\mu_{n,ac}\left(\bigcup_{0\le i\le n-1}\bigcup_{1\le j\le k(i)}B_{f^i(D)}(x_{i,j},r_1/4)\right) \geq \eta_0, $$
where $\eta_0=\tau \theta {\rm Leb}_D(\Lambda_{\lambda_1,1} \cap D_0)$.

Given $i\ge 0$, for any measurable sets $A,B\subset f^i(D)$, we have that
$$\frac{{\rm Leb}_{f^i(D)}(A)}{{\rm Leb}_{f^i(D)}(B)}=\frac{\int_{f^{-i}(A)}|{\rm det}( Df^i)|d{\rm Leb}(D)    }         {\int_{f^{-i}(B)}|{\rm det}(Df^i)|d{\rm Leb}(D)}= \frac{|{\rm det}(Df^i(\xi_A))|{\rm Leb}_D(f^{-i}(A))}{|{\rm det}(Df^i(\xi_B))|{\rm Leb}_D(f^{-i}(B))}$$
for some $\xi_A\in f^{-i}(A)$ and $\xi_B\in f^{-i}(B)$.

By Proposition~\ref{pro:bounded distortion}, if we take $A=B_{f^i(D)}(x,r_1/4)\setminus B_{f^i(D)}(x,\frac{r_1}{4}-\varepsilon)$ and $B=B_{f^i(D)}(x,r_1/4)$, we have that
$$ \frac{f_*^i{\rm Leb}_D\left(B_{f^i(D)}(x,r_1/4)\setminus B_{f^i(D)}(x,\frac{r_1}{4}-\varepsilon)\right)} {f_*^i{\rm Leb}_D(B_{f^i(D)}(x,r_1/4))}=\frac{{\rm Leb}_D(f^{-i}(A))}{{\rm Leb}_D(f^{-i}(B))}\le \mathcal{K}\frac{{\rm Leb}_{f^i(D)}(A)}{{\rm Leb}_{f^i(D)}(B)},$$
Due to the fact that $$\frac{{\rm Leb}_{f^i(D)}\left(B_{f^i(D)}(x,r_1/4)\setminus B_{f^i(D)}(x,\frac{r_1}{4}-\varepsilon)\right)}{{\rm Leb}_{f^i(D)}(B_{f^i(D)}(x,r_1/4))} $$ can be arbitrary small by reducing $\varepsilon$, we have that for $\eta=\eta_0/2$, there is $\widetilde{\varepsilon}_0 >0$ small enough such that for any $\varepsilon\in [0,\widetilde{\varepsilon}_0)$, one obtains $\mu_{n.ac.\varepsilon}(\bigcup_{0\le i\le n-1}K_{i,\varepsilon})\ge \eta$. The proof is complete.
\end{proof}

Now let $K_{\infty}=\bigcap_{n=1}\overline{\bigcup_{j\ge n}K_j}$ which is the accumulation points of $\{K_j\}_{j\ge 1}$, let $x_{\infty}$ be an accumulation point of $\{x_{n,j(n)}\}$ for some $j(n)$, up to considering the subsequences we may suppose $x_{n,j(n)}\rightarrow x_{\infty}$. As we have shown, disks $\{B_{f^n(D)}(x_{n,j(n)},r_1/4), n\ge 1\}$ are all tangent to the $F$-direction cone field of fixed width $a$ with uniform size, and they have the uniform H\"{o}lder curvature when $n$ large enough by applying Proposition \ref{Holder curvature} (Recall $n$ is the hyperbolic time). Therefore, Ascoli-Arzela theorem ensures that there exists a disk $B(x_{\infty})$ of radius $r_1/4$ around $x_{\infty}$ such that $B_{f^n(D)}(x_{n,j(n)},r_1/4)$ converges to $B(x_{\infty})$ in the $C^1$ topology, then $B(x_{\infty})\subset K_{\infty}$.

We will prove certain properties of accumulation points and corresponding disks.

\begin{lemma}\label{lem:unstable}Let $x_{\infty}$ be an accumulation point of $\{x_{n,j(n)}\}$ for some $j(n)$, and suppose $B(x_{\infty})$ is the accumulation disk, then we have
\begin{enumerate}
\item $K_{\infty}\subset K$, and in particular, $x\in B(x_{\infty})\subset K$;

\item the subspace $F(x_\infty)$ is uniformly expanding in the following sense:
$$
\|{Df}^{-k}/F(x_\infty)\|\leq \lambda_2^{k/2} ~~~\text{for every}~~ k~ \geq~ 1 ;
$$
\item $B(x_\infty)$ is contained in the corresponding strong unstable manifold $\mathcal{W}^u_{loc}(x_\infty)$:
$$d(f^{-k}(x_\infty),f^{-k}(y))\leq \lambda_2^{k/2} d(x_\infty,y), ~~~ \forall  y\in B(x_\infty);$$
\item $B(x_\infty)$ is tangent to  $F(y)$ for  every point $y\in B(x_\infty)$.
\end{enumerate}
\end{lemma}

\begin{proof}

By the construction, one observes that $K_j\subset f^{\ell}(U)$ for any $j\ge {\ell}$. Then $\bigcup_{j\ge {\ell}}K_j\subset f^{\ell}(U)$. This implies $$\overline{\bigcup_{j\ge {\ell}}K_j}\subset \overline{f^{\ell}(U)} \subset f^{{\ell}-1}(U).$$ Therefore, $$K_{\infty}=\bigcap_{{\ell}\in \mathbb{N}}\overline{\bigcup_{j\ge {\ell}}K_j}\subset \bigcap_{{\ell}\in \mathbb{N}}f^{{\ell}-1}(U)=K.$$ Then $x\in B(x_{\infty})\subset K_{\infty}\subset K$. We obtain conclusion $(1)$.

Next we will check the last three conclusions. By construction and Proposition \ref{diskproperty}, we have the following:
\begin{itemize}
\item $\prod_{l=0}^{k-1}\|Df^{-1}/F(f^{-l}(x_{n,j(n)})\|\le \lambda_2^k$ for every $1\leq k \leq n$ and every $n$;

\item for every $k\geq 1$, $f^{-k}$ is a $\lambda_2^{k/2}$ contraction on $B_{f^n(D)}(x_{n,j(n)},r_1/4)$ for every $n$, i.e., $d(f^{-k}x_{n,j(n)},f^{-k}y)\leq \lambda_2^{k/2}d(x_{n,j(n)},y)$ for every $0 \leq k\leq n$, whenever $y$ is contained in $B_{f^n(D)}(x_{n,j(n)},r_1/4)$;

\item disks $\{B_{f^n(D)}(x_{n,j(n)},r_1/4)\}$ are contained in the corresponding $F$-direction cone field and angles between $F$ and the tangent spaces of these disks are exponentially contracted as $n$ increasing.

\end{itemize}

Passing to the limit, we know $(2),(3),(4)$ are true.

\end{proof}

\begin{definition}
A \emph{fake $F$-cylinder} at some point $y$ is a set $\exp_y(\varphi(X\times D_0))$, where $X\subset {\mathbb R}^{\dim E}$ is a compact set, $D_0\subset{\mathbb R}^{\dim F}$ is the unit ball such that for each $x\in X $, $\varphi_x:~D_0\to E$ is a $C^{1+\xi}$ map
\begin{itemize}

\item $\exp_y(\varphi_x(D_0))$ tangent to the cone field ${\cal C}_a^F$.

\end{itemize}

If in addition, we have that

\begin{itemize}

\item $\exp_y(\varphi_x(D_0))$ is a local unstable manifold.

\item the intersection of $\exp_y(\varphi_x(D_0))$ and $\exp_y(\varphi_z(D_0))$ is relatively open in each one for any $x,z\in X$.

\end{itemize}

then we say that $\exp_y(\varphi(X\times D_0))$ is a \emph{$F$-cylinder}.

$\{\exp_y(\varphi_x(D_0)\}_{x\in X}$ is called the \emph{canonical partition} of the $F$-cylinder.

\end{definition}

\begin{definition}

For two finite Borel measures $\nu_1$ and $\nu_2$, we denote $\nu_1\prec\nu_2$ if for any measurable set $A$, we have $\nu_1(A)\le \nu_2(A)$.

\end{definition}

\begin{proposition}\label{partabsolutelycontinuous}

There is a measure $\mu_{ac}\prec\mu$ and and $F$-cylinder $L_\infty$ such that $\mu_{ac}(L_\infty)> 0$ and the conditional measure of $\mu_{ac}$ associated to the canonical partition ${\cal L}_\infty$
is absolutely continuous with respect to the Lebesgue measure for almost every $\gamma\in{\cal L}_\infty$.
\end{proposition}

\begin{proof}

Let $\{n_k\}$ be a subsequence such that $\{\mu_{n_k}\}$ accumulates. By taking a subsequence if necessary, one can assume that $\{\mu_{n_k,ac}\}$ accumulates. Set $\mu_{ac}=\lim_{n\to\infty}\mu_{n,ac}$. We have $\mu_{ac}(\overline{U})\geq \limsup_{k \rightarrow \infty}\mu_{n_k,ac}(\overline{U})\geq\eta>0$, and then $\mu_{ac}(K_{\infty})\ge \eta$, since ${\rm supp}(\mu_{ac})\subset K_{\infty}$.

For $\varepsilon>0$ small, take
$$K_{\infty,\varepsilon}=\bigcap_{n\in\mathbb N}{\overline {\bigcup_{j\ge n}K_{j,\varepsilon}}}.$$
We have ${\rm supp}(\mu_{ac,\varepsilon})\subset K_{\infty,\varepsilon}$. Take a point $y\in {\rm supp}(\mu_{ac,\varepsilon})$. Then for $\delta>0$, we have $\mu_{ac}(K_{\infty}\cap B(y,\delta))\ge \mu_{ac,\varepsilon}(K_{\infty,\varepsilon}\cap B(y,\delta))>0$. By construction we have that $K_{\infty,0}\cap B(y,\delta)$ is an $F$-cylinder if we take $\delta\ll \varepsilon$, where $B(y,\delta)$ is a small open neighborhood of $y$ with radius $\delta$. Set
$$L_\infty=K_{\infty,0}\cap U(x,\delta)=\bigcup_{x\in X_\infty}\exp_y(\varphi_x(D_0)),$$
where $X_\infty=\{x\in E(y):x\in \exp_y^{-1}(\gamma),\gamma~\textrm{is an unstable leaf in~}K_{\infty,0}\}.$

Define $X_n=\{x\in E(y):x\in \exp_y^{-1}(B_{f^n(D)}(x_{n,j(n)},r_1/4)),~\textrm{for some~}x_{n,j(n)}\in f^n(D),~\textrm{where~}n~\textrm{is a~}\lambda_2\textrm{-hyperbolic time for~}f^{-n}(x_{n,j(n)})\}$. Notice that $X_n$ may have non-empty intersection with $X_\infty$ or $X_m$ for $m\neq n$.

By the construction, we have that $\mu_{ac}\prec\mu$. Now we need to show that the conditional measure of $\mu_{ac}$ associated to the canonical partition of ${\cal L}_\infty$ is absolutely continuous with respect to the Lebesgue measure for almost every $\gamma\in{\cal L}_\infty$.

Define
$$L_n=\left(\bigcup B_{f^n(D)}(x_{n,j(n)},r_1/4)\right)\cap B(y,\delta).$$
Notice that $L_n$ can be identified to be a fake $F$-cylinder as ${\exp}_y\varphi(X_n\times D_0)$.

Let $\widehat{L}=\bigcup_{0\leq i \le \infty}L_i\times \{i\}$, and $\widehat{\mu}_{n,ac}$ be
$$
\widehat{\mu}_{n,ac}(\bigcup_{i=0}^{n-1}B_{i}\times \{i\})=\frac{1}{n}\sum_{i=0}^{n-1}f_{\ast}^{i} {\rm Leb}_D(B_i),
$$
where $B_i\subset L_i$ is a measurable set.

We can define a limit in $\widehat{L}$ by the following way: we define $\lim_{n\to\infty}(x_n,m(n))=(x_0,n_0)$ if and only if $\lim_{n\to\infty}x_n=x_0$ in the Riemannian metric of the manifold $M$, and one of the following cases occurs
\begin{itemize}

\item $n_0=\infty$, $m(n)=\infty$ and $x_0,x_n\in L_\infty$ for $n$ large enough;

\item $n_0=\infty$, $\lim_{n\to\infty}m(n)=\infty$ and $x_n\in L_{m(n)}$;

\item $n_0$ if finite, for $n$ large enough, $m(n)=n_0$, $x_0,x_n\in L_{n_0}$.

\end{itemize}

This limit gives a topology on $\widehat L$, and under this topology, $\widehat L$ is a compact space.

\bigskip

The fake $F$-cylinder
$$\exp_{y}\left(\left(\bigcup_{1\le n\le\infty}X_n\right)\times D_0\right)$$
gives a measurable partition on $\widehat L$.

By the Proposition~\ref{pro:bounded distortion}, there is a constant $\mathcal{C}>0$ such that for each measurable set $B\subset D_0$, for each $n\in{\mathbb N}$, we have
$$\frac{1}{\mathcal{C}}\frac{{\rm Leb}(B)}{{\rm Leb}(D_0)}\le \frac{{\widehat \mu}_{n,ac}(\cup_{i=0}^{n-1}X_i\times B)}{{\widehat \mu}_{n,ac}(\cup_{i=0}^{n-1}X_i\times D)}\le {\mathcal{C}}\frac{{\rm Leb}(B)}{{\rm Leb}(D_0)}.$$

By using the dominated convergence theorem, for almost every disk in the $F$-cylinder $L_\infty$, the conditional measure of $\mu_{ac}$ is absolutely continuous with respect to the Lebesgue measure.

\end{proof}
\subsection{Existence of SRB measure and Physical measure}

For each $x\in M$, one can consider the measures
$$\mu_{x,n}=\frac{1}{n}\sum_{i=0}^{n-1}\delta_{f^i(x)}.$$
The set $\Sigma$ is defined to be: $x\in\Sigma$ if and only if $\lim_{n\to\infty}\mu_{x,n}$ exists and is ergodic. From Ergodic Decomposition theorem \cite[Chapter II.6]{Man87}, one knows that $\Sigma$ has total probability and if one denotes $\mu_{x}=\lim_{n\to\infty}\mu_{x,n}$, then for any bounded measurable function $\psi$ and any invariant measure $\nu$, one has $x\mapsto \int \psi d\mu_x$ is measurable and
$$\int \psi  d\nu=\int\int \psi d\mu_xd\nu(x).$$

%
%

\begin{lemma}\label{Lem:mu-disintegration}

There is a set $Z_\infty\subset L_\infty\cap\Sigma$ such that $\mu(Z_\infty)>0$ and the conditional measures of $(\mu|Z_\infty)$ on the unstable manifolds are absolutely continuous with respect to the Lebesgue measures on these manifolds.

\end{lemma}

\begin{proof}
We consider a family of measurable sets $\cal A$ such that for each $A\in{\cal A}$, we have $A\subset L_\infty\cap \Sigma$ and ${\rm Leb}_\gamma(\gamma\cap A)=0$ for each leaf $\gamma\in{\cal L}_\infty$. We can find such a measurable set $A_\infty$ such that
$$\mu(A_\infty)=\max_{A\in{\cal A}}\mu(A).$$
Such a maximal exists because if we have a sequence of measurable sets $\{A_n\}$ such that $\lim_{n\to\infty}\mu(A_n)=\sup_{A\in{\cal A}}\mu(A)$, then we take $A_{\infty}=\cup_{n=1}^{\infty}A_n$. By the definition of $\cal A$, we have $A_{\infty}\in\cal A$ , then $\mu_{ac}(A_{\infty})=0$, for the conditional measures of $\mu_{ac}$ along  the leaves of $\mathcal{L}_{\infty}$ are absolutely continuous with respect to Lebesgue as we proved in Proposition \ref{partabsolutelycontinuous}.

Set $Z_\infty=L_\infty\cap \Sigma\setminus A_\infty$. Since $\mu_{ac}(Z_\infty)=\mu_{ac}(L_{\infty})>0$, we have $\mu(Z_\infty)>0$. For any measurable set $A\subset Z_\infty$ satisfying ${\rm Leb}_{\gamma}(A\cap \gamma)=0$ for almost every $\gamma\in{\cal L}_\infty$, by the definition of $A_\infty$ and $Z_\infty$, we have $(\mu|Z_\infty)(A)=0$. This implies that  $(\mu|Z_\infty)$ has absolutely continuous conditional measures on the unstable manifolds.

\end{proof}

%
%

\begin{lemma}\label{Lem:localconstant}

By reducing $\Sigma$ if necessary, for every two points $x,y\in\Sigma\cap \gamma$ for some unstable manifold $\gamma$, we have that $\mu_x=\mu_y$.

\end{lemma}

\begin{proof}

According to Birkhoff Ergodic theorem, by reducing $\Sigma$ if necessary, one can assume that for any $x\in\Sigma$, $\lim_{n\to\infty}1/n\sum_{i=0}^{n-1}\delta_{f^{-i}(x)}$ exists and equals to $\mu_x$.
For any $x,y\in\Sigma\cap \gamma$, one has $\lim_{n\to\infty}1/n\sum_{i=0}^{n-1}\delta_{f^{-i}(x)}=\lim_{n\to\infty}1/n\sum_{i=0}^{n-1}\delta_{f^{-i}(y)}$ by the definitions. This implies $\mu_x=\mu_y$.
\end{proof}

Denote by $\mathcal{P}=\{\gamma \cap Z_{\infty}:\gamma \in \mathcal{L}_{\infty}\}$ and ${\cal Q}=\{ Q\subset Z_{\infty}:~x,y\in Q~\textrm{if~and~only~if~}\mu_x=\mu_y\}$ the two measurable partitions of $Z_{\infty}$, then from Lemma \ref{Lem:localconstant} we have $\mathcal{Q}\prec\mathcal{P}$ which means $\mathcal{P}$ is finer than $\mathcal{Q}$. Also let $\pi_{\cal{P}}:Z_{\infty}\rightarrow \cal{P}$ and $\pi_{\cal{Q}}:Z_{\infty}\rightarrow \cal{Q}$ be the projections.

For every measurable subset $A\subset Z_\infty$, by the Ergodic Decomposition theorem we mentioned above, take $\psi=\chi_A$, we obtain
$$
\mu(A)=\int_{\Sigma} \mu_x(A)d\mu(x)
$$
and
$$
\mu_x(A)=\int \chi_{A}d\mu_x=\lim_{n\rightarrow +\infty}\frac{1}{n}\sum_{i=0}^{n-1}\chi_{A}(f^i(x))
$$
$\mu$-almost everywhere, where $\chi_{A}$ denotes the characteristic function of measurable set A. By Poincar\'{e}'s recurrence theorem we know $\mu$ almost every point $z\in Z_\infty$ has infinitely many  return times. Let $k(z)$ be the smallest integer such that $f^{-k(z)}(z)\in Z_\infty$. One knows
$$
\mu(A)=\int_{Z_\infty}k(z)\mu_z(A)d\mu(z)
$$
for any measurable subset $A$ of $Z_\infty$,
where one uses the fact $\mu_z=\mu_{f^i z}$ for every $i\in \mathbb{Z}$.

Recall the definition and properties of conditional expectation. Given two $\sigma$-algebras $\mathcal{B}_1$, $\mathcal{B}_2$ with the property $\mathcal{B}_2\subset \mathcal{B}_1$, that is to say $\mathcal{B}_2$ is the sub-$\sigma$-algebra of $\mathcal{B}_1$. Consider measurable space $(X,\mathcal{B}_1,\mu)$, one can define a conditional expectation operator $E(\cdot/\mathcal{B}_2):L^1(X,\mathcal{B}_1,\mu)\rightarrow L^1(X,\mathcal{B}_2,\mu)$ such that for every function $\phi\in L^1(X,\mathcal{B}_1,\mu)$, $E(\phi/\mathcal{B}_2)$ is the $\mu$.a.e. unique $\mathcal{B}_2$-measurable function with
$$
\int_A \phi d\mu=\int_A E(\phi/\mathcal{B}_2)d\mu,~~~\text{for every}~~A\in \mathcal{B}_2.
$$
For every $\phi\in L^1(X,\mathcal{B}_1,\mu)$ and $\mathcal{B}_2$-bounded measurable function $\psi$, we have $E(\phi \psi/\mathcal{B}_2)=\psi E(\phi /\mathcal{B}_2)$. That is to say
$$
\int_A \phi \psi d\mu= \int_A \psi E(\phi/\mathcal{B}_2)d\mu,~~~\text{for every}~~A\in \mathcal{B}_2.
$$

Consider the sub-$\sigma$-algebra $\cal{B}(\cal{Q})$ of the original $\cal{B}$ which is generated by the measurable partition $\cal{Q}$. Then there exists a unique conditional expectation $\ell$ of the function $k$ which is a measurable function defined on $\cal{B}(\cal{Q})$ and $\ell$ is constant on each element of $\cal{Q}$. Moreover, as $z\mapsto \mu_z(A)$ is $\cal{B}(\cal{Q})$-bounded measurable functions for every measurable set $A$, by Ergodic Decomposition theorem, we have
$$
\int_E k(z)\mu_z(A)d\mu=\int_E \ell(z)\mu_z(A)d\mu,~~~\text{for every}~~\mathcal{B}(\mathcal{Q})\text{-measurable set}~E.
$$
We can define $\mu_{Q}=\mu_z$ and $\ell(Q)=\ell(z)$ for some $z\in Q$, every $Q\in \cal{Q}$. They are well defined as $\mu_z$ and $\ell(z)$ are constant on each element of $\cal{Q}$. Thus,
$$
\int \ell(z)\mu_z(A)d\mu=\int \ell(Q)\mu_Q(A)d\widehat\mu_{\cal{Q}},
$$
where $\widehat\mu_{\cal{Q}}$ is the quotient measure defined by $\widehat\mu_{\cal{Q}}(B)=\pi_{\cal{Q}}^{-1}(B)$ for any $B\subset \cal{Q}$.
So,
$$
\mu(A)=\int \ell(Q)\mu_Q(A)d\widehat\mu_{\cal{Q}}, ~~\text{for every measurable subset}~A\subset Z_{\infty}.
$$
We have the following claim:
\begin{Claim}\label{lemma;conditional measures}
 $\{\ell(Q)\mu_Q\}_{Q\in \cal{Q}}$ is the family of conditional measures of $\mu$ with respect to the measurable partition $\cal{Q}$.
\end{Claim}

\begin{proof} Without loss of generality, by contradiction, we may assume that there is a subset $B\subset \mathcal{Q}$ with positive $\widehat\mu_{\cal{Q}}$ measure such that
$$
\ell(Q)\mu_Q(Q)>1 ~~\text{for any}~~Q\in B.
$$
By the definition of $\widehat\mu_{\cal{Q}}$, we have
\begin{eqnarray*}
 \widehat\mu_{\cal{Q}}(B)=\mu (\pi_{\mathcal{Q}}^{-1}(B))& = &\int\ell(Q)\mu_Q(\pi_{\mathcal{Q}}^{-1}(B))d\widehat\mu_{\cal{Q}} \\
   & = & {\int}_B \ell(Q)\mu_Q(\pi_{\mathcal{Q}}^{-1}(B))d\widehat\mu_{\cal{Q}} \\
   & + &  {\int}_{\mathcal{Q}\setminus B} \ell(Q)\mu_Q(\pi_{\mathcal{Q}}^{-1}(B))d\widehat\mu_{\cal{Q}}.
\end{eqnarray*}
Observe that $\mu_Q(\pi_{\mathcal{Q}}^{-1}(B))=0$ for any $Q\in \mathcal{Q}\setminus B$, this is because of
$\pi_{\mathcal{Q}}^{-1}(B)\subset Z_{\infty}\setminus Q$ for every $Q\in \mathcal{Q}\setminus B$ and the fact $\mu_Q(Z_{\infty}\setminus Q)=0$ for every $Q\in \mathcal{Q}$. On the other hand, $Q\in B$ implies $\mu_Q(\pi_{\mathcal{Q}}^{-1}(B))=\mu_Q(Q)$, all these together we know
\begin{eqnarray*}
  \widehat\mu_{\cal{Q}}(B) &=& {\int}_B \ell(Q)\mu_Q(\pi_{\mathcal{Q}}^{-1}(B))d\widehat\mu_{\cal{Q}} \\
  &=& {\int}_B \ell(Q)\mu_Q(Q)d\widehat\mu_{\cal{Q}} \\
  &>& \widehat\mu_{\cal{Q}}(B).
\end{eqnarray*}
This gives a contradiction, which completes the proof of the claim.

\end{proof}

\begin{lemma}\label{lem:ergodicmeasure}
There exists some point $z\in Z_{\infty}$ such that $\mu_z(Z_{\infty})>0$ and $\mu_z$ has absolutely continuous conditional measures along the leaves of $\mathcal{L}_{\infty}$.
\end{lemma}
\begin{proof}
For every $Q\in \cal{Q}$, let $\{\mu_{Q,P}: P\in \mathcal{P},P\subset Q \}$ be the family of conditional measures of $\mu_{Q}$ with respect to the finer  partition $\cal{P}$ restrict to $Q$. Denote by $\widehat{\mu}_{Q,\mathcal{P}}$ the quotient measure of $\mu_{Q}$ with respect to the partition $\mathcal{P}$ restricted to $Q$, then by definition for every measurable set $A$ we have
$$
\mu_{Q}(A)=\int\mu_{Q,P}(A)d\widehat{\mu}_{Q,\mathcal{P}},
$$
which implies
$$
\ell(Q)\mu_{Q}(A)=\int\mu_{Q,P}(A)d\ell(Q)\widehat{\mu}_{Q,\mathcal{P}}.
$$
So $\{\mu_{Q,P}\}$ are conditional measures of $\ell(Q)\mu_{Q}$ with respect to partition $\cal{P}$ restricted to $Q$. If we denote by $\{\mu_P\}_{P\in \cal{P}}$ as the conditional measures of $\mu$ with respect to the finer partition $\cal{P}$, as we have shown in the claim above, $\{\ell(Q)\mu_Q\}_{Q\in \cal{Q}}$ are the conditional measures of $\mu$ with respect to the measurable partition $\cal{Q}$. Therefore, by the essential uniqueness of Rokhlin decomposition we have
$\mu_P=\mu_{Q,P}$ for $\widehat\mu_{\cal{Q}}$-almost every $Q\in \cal{Q}$ and $\widehat\mu_{Q,\cal{P}}$-almost every $P\in \cal{P}$ with $P\subset Q$. By definition of $\mu_Q$, which is equivalent to say $\mu_P=\mu_{z,P}$ for $\mu$-almost every $z\in Z_{\infty}$ and $\widehat{\mu}_z$-almost every $P$, where $\widehat{\mu}_z$ represents for the quotient measure of $\mu_z$ with respect to partition $\cal{P}$.

Since we have
$$
\int_{Z_{\infty}} k(z)\mu_z(Z_{\infty})d\mu= \mu(Z_{\infty})>0,
$$
there exists a subset $Z_1 \subset Z_{\infty}$ such that $\mu_z(Z_{\infty})>0$ for every $z\in Z_1$. Furthermore, as we have shown in Lemma \ref{Lem:mu-disintegration}, $\mu_P$ is absolutely continuous with respect to Lebesgue measure for almost every $P$. Here one should notice that for every $P\in \cal{P}$, we have $P=\gamma \cap Z_{\infty}$, for some $\gamma \in \cal{L}_{\infty}$ by the construction of $\cal{P}$. Then by the argument above we obtain a set $Z_2$ with full $(\mu|Z_{\infty})$ measure such that for every $z\in Z_2$,
$\mu_{z,P}$ is absolutely continuous with respect to Lebesgue measure for $\widehat{\mu}_z$-almost every $P\in \cal{P}$. So if one takes some point $z\in Z_1\cap Z_2$, then it satisfies the requirement of this lemma.

\end{proof}

\begin{proof}[Proof of Theorem A] Take $z\in Z_{\infty}$ as in Lemma \ref{lem:ergodicmeasure}, then $\mu_z(L_{\infty})>0$.
For every regular points $y$ we have $\lim_{n\rightarrow \infty}\frac{1}{n}\log \|Df^{-n}/F(y)\|\leq \frac{1}{2}\log\lambda_2$,
which can be concluded from Lemma~\ref{lem:unstable}. That is to say there exists a set with positive $\mu_z$-measure such that all the points there have dim$F$ Lyapunove exponents
larger than $-\frac{1}{2}\log \lambda_2$,  so we have that $\mu_z$  has dim$F$ positive Lyapunov exponents by ergodicity.
By assumption, we know all the Lyapunov exponents along $E$-direction are non positive for $\mu_z$ almost every point. So, by  Pesin theory (see more in \cite{bp02} for instance)
we obtain that $\mu_z$-almost every point $x$ has a local unstable manifold. Furthermore, since the disks $\gamma \in \mathcal{L}_{\infty}$  are contained in
the local unstable manifolds. Using the ergodicity and absolute continuity property proved in Lemma \ref{lem:ergodicmeasure},
we have that the ergodic measure $\mu_z$ has absolutely continuous conditional measures on unstable manifolds. This ends the proof of Theorem~\ref{Theo-attractor}.
\end{proof}


\begin{proof}[Proof of Corollary 1]
The condition
$$
\liminf_{n\rightarrow \infty}\frac{1}{n}\log\|Df^n/E(x)\|<0.
$$
on a total probability set implies that $E$ is uniformly contracted by the work of Cao in \cite{cao03}.  Since we have
found an ergodic SRB measure $\mu$, then
$\mu$ is a Physical measure by using the absolute continuity of stable foliation. One can see \cite{y02} for more details.
\end{proof}

\begin{proof}[A sketch of the proof of Theorem~\ref{Thm;submanifold}]

By the assumption, as in the proof of Theorem~\ref{Theo-attractor}, mainly applying Lemma \ref{Lem:infinitePliss} we know that there exist $\lambda_1\in(0,1)$ and some $j\in \mathbb{N}$ such that the following set
$$
\left\{ x\in f^j(D)\cap \Lambda_{\lambda_1,1}:~T_x D=F(x) \right\}$$
has positive Lebesgue measure in $f^j(D)$. Then we take a Lebesgue density point of the above set and a small sub-disk around this point. By following the the proof of Theorem \ref{Theo-attractor}, we know the existence of SRB measures.

\end{proof}

\vskip 5pt

\noindent Zeya Mi

\noindent School of Mathematical Sciences

\noindent Soochow University, Suzhou, 215006, P.R. China

\noindent mizeya@163.com, mi.zeya@northwestern.edu

\vskip 5pt

\noindent Yongluo Cao

\noindent School of Mathematical Sciences

\noindent Soochow University, Suzhou, 215006, P.R. China

\noindent ylcao@suda.edu.cn

\vskip 5pt

\noindent Dawei Yang

\noindent School of Mathematical Sciences

\noindent Soochow University, Suzhou, 215006, P.R. China

\noindent yangdw1981@gmail.com, yangdw@suda.edu.cn

\end{document}